\documentclass[12pt]{amsart}

\usepackage{amsmath}
\usepackage{amssymb}
\usepackage{amsfonts}
\usepackage{amsthm}
\usepackage{enumerate}
\usepackage{hyperref}
\usepackage{color}

\theoremstyle{plain}
\newtheorem{thm}{Theorem}[section]

\newtheorem{lem}[thm]{Lemma}
\newtheorem{cor}[thm]{Corollary}

\theoremstyle{definition}
\newtheorem{dfn}[thm]{Definition}

\newtheorem{exmp}[thm]{Example}

\newtheorem{dfns-rems}[thm]{Definitions and Remarks}
\newtheorem{notas-rems}[thm]{Notations and Remarks}
\newtheorem{exmps-rems}[thm]{Examples and Remarks}

\begin{document}


\title[Regularity of symbolic powers of cover ideals]{Regularity of symbolic powers of cover ideals of graphs}


\author[S. A. Seyed Fakhari]{S. A. Seyed Fakhari}

\address{S. A. Seyed Fakhari, School of Mathematics, Statistics and Computer Science,
College of Science, University of Tehran, Tehran, Iran.}

\email{aminfakhari@ut.ac.ir}


\begin{abstract}
Let $G$ be a graph which belongs to either of the following classes: (i) bipartite graphs, (ii) unmixed graphs, or (iii) claw--free graphs.  Assume that $J(G)$ is the cover ideal $G$ and $J(G)^{(k)}$ is its $k$-th symbolic power. We prove that$$k{\rm deg}(J(G))\leq {\rm reg}(J(G)^{(k)})\leq (k-1){\rm deg}(J(G))+|V(G)|-1.$$We also determine families of graphs for which the above inequalities are equality.
\end{abstract}






\thanks{}


\maketitle


\section{Introduction} \label{sec1}

Let $I$ be a homogeneous ideal in the polynomial ring $S = \mathbb{K}[x_1, \ldots, x_n]$. Suppose that the minimal free resolution of $I$ is given by
$$0\rightarrow \cdots \rightarrow \bigoplus_j S(-j)^{\beta _{1,j}(I)}\rightarrow \bigoplus_j S(-j)^{\beta _{0,j}(I)}\rightarrow I\rightarrow 0.$$
The Castelnuovo-Mumford regularity (or simply, regularity) of $I$, denoted by ${\rm reg}(I)$, is defined as
$${\rm reg}(I)={\rm max}\{j-i\mid \beta _{i,j}(I)\neq 0\},$$
and is an important invariant in commutative algebra and algebraic geometry.

Computing and finding bounds for the regularity of powers of a monomial ideal have been studied by a number of researchers (see for example \cite{ab}, \cite{abs}, \cite{b''}, \cite{bht}, \cite{c'}, \cite{cht}, \cite{htt}, \cite{jns} \cite{msy}). This work is motivated by a recent paper of Hang and Trung \cite{ht1}. In that paper, the authors study the regularity of powers of cover ideals of the so-called unimodular hypergraphs. It is well-known that the class of unimodular hypergraphs includes the family of bipartite graphs. Restricting to this family of graphs, their Theorem 3.3 (see also \cite[Corollary 3.4]{ht1}) says that if $G$ is a bipartite graph with cover ideal $J(G)$ (see Definition \ref{cover}), then there is a non-negative integer $e\leq |V(G)|-{\rm deg}(J(G))-1$ such that$${\rm reg}(J(G)^k)=k{\rm deg}(J(G))+e,$$for every integer $k\geq |V(G)|+2$. Here, ${\rm deg}(J(G))$ denotes the maximum degree of the minimal monomial generators of $J(G)$. Consequently, for every integer $k\geq |V(G)|+2$, we have $${\rm reg}(J(G)^k)\leq (k-1){\rm deg}(J(G))+|V(G)|-1.$$It is natural to ask wether the above inequality is valid for every non-negative integer $k$. In Theorem \ref{bip}, we give a positive answer to this question. Note that by \cite[Corollary 2.6]{grv}, the ordinary and the symbolic powers of the cover ideal of a bipartite graph are the same. Therefore, it is reasonable to study the regularity of symbolic powers of cover ideals. This will be done in Section \ref{sec3}. The most general result of this paper is Theorem \ref{reg}. Its statement is as follows. Let $\mathcal{H}$ be a family of graphs such that (i) for every graph $G\in \mathcal{H}$ and every vertex $x\in V(G)$, the graph $G\setminus N_G[x]$ belongs to $\mathcal{H}$, and (ii) every $G\in \mathcal{H}$ which has no isolated vertex, admits a minimal vertex cover with cardinality at least $\frac{|V(G)|}{2}$. We show in Theorem \ref{reg} that for every graph $G\in \mathcal{H}$ and every integer $k\geq 1$,
\[
\begin{array}{rl}
k{\rm deg}(J(G))\leq {\rm reg}(J(G)^{(k)})\leq (k-1){\rm deg}(J(G))+|V(G)|-1.
\end{array} \tag{1} \label{1}
\]
It is easy to see that the class of bipartite graphs satisfies the assumption of Theorem \ref{reg}, which implies the above mentioned result (see Theorem \ref{bip}). It follows from \cite{gv} that the class of unmixed graphs also satisfies the assumption of Theorem \ref{reg}. Hence, the inequalities (\ref{1}) are true for every unmixed graph too (see Theorem \ref{unmixed}). In Theorem \ref{claw}, we show that the class of claw--free graphs also satisfies the assumption of Theorem \ref{reg}. This means that the inequalities (\ref{1}) are true for claw--free graphs too. In Corollaries \ref{eq1} and \ref{eq2}, we determine the families of graphs for which the inequalities (\ref{1}) are equality, showing that theses inequalities are sharp.


\section{Preliminaries} \label{sec2}

In this section, we provide the definitions and basic facts which will be used in the next section. We refer the reader to \cite{hh} for undefined terminologies.

Let $G$ be a simple graph with vertex set $V(G)=\big\{x_1, \ldots,
x_n\big\}$ and edge set $E(G)$ (by abusing the notation, we identify the vertices of $G$ with the variables of $S$). For a vertex $x_i$, the {\it neighbor set} of $x_i$ is $N_G(x_i)=\big\{x_j\mid \{x_i, x_j\}\in E(G)\big\}$ and we set $N_G[x_i]=N_G(x_i)\cup \{x_i\}$ and call it the {\it closed neighborhood} of $x_i$. The {\it degree} of $x_i$, denoted by ${\rm deg}_G(x_i)$ is the cardinality of $N_G(x_i)$. A vertex of degree one is called a {\it leaf}. An edge $e\in E(G)$ is a {\it pendant edge}, if it is incident to a leaf. For every subset $A\subset V(G)$, the graph $G\setminus A$ is the graph with vertex set $V(G\setminus A)=V(G)\setminus A$ and edge set $E(G\setminus A)=\{e\in E(G)\mid e\cap A=\emptyset\}$. A subset $W$ of $V(G)$ is said to be an {\it independent subset} of $G$ if there are no edges among the vertices of $W$. The cardinality of the largest independent subset of $G$ is the {\it independence number} of $G$ and is denoted by $i(G)$.
A subset $C$ of $V(G)$ is called a {\it vertex cover} of $G$ if every edge of $G$ is incident to at least one vertex of $C$. A vertex cover $C$ is called a {\it minimal vertex cover} of $G$ if no proper subset of $C$ is a vertex cover of $G$. Note that $C$ is a minimal
vertex cover if and only if $V(G)\setminus C$ is a maximal independent set. The graph $G$ is called {\it unmixed} if all
minimal vertex covers of $G$ have the same cardinality. The graph $G$ is said to be {\it complete} if each pair of vertices of $G$ are adjacent by an edge. The graph $G$ is {\it bipartite} if there exists a partition $V(G)=A\cup B$ such
that each edge of $G$ is of the form $\{x_i,x_j\}$ with $x_i\in A$ and $v_j\in B$. If moreover, every vertex of $A$ is adjacent to every vertex of $B$, then we say that $G$ is a {\it complete bipartite} graph and denote it by $K_{a,b}$, where $a=|A|$ and $b=|B|$. The graph $K_{1,3}$ is called a {\it claw} and the graph $G$ is said to be {\it claw--free} if it has no claw as an induced subgraph.

We now define the main objective of this paper.

\begin{dfn} \label{cover}
Let $G$ be a graph with $n$ vertices. The {\it cover ideal} of $G$, denoted by $J(G)$ is a squarefree monomial ideal of $S$ which is defined as follows.$$J(G)=\big(\prod_{x_i\in C}x_i \mid C \ {\rm is \ a \ minimal \ vertex \ cover \ of} \ G\big)$$It is well-known that$$J(G)=\bigcap_{\{x_i, x_j\}\in E(G)}(x_i, x_j).$$In other words, $J(G)$ is the Alexander dual of the so-called edge ideal of $G$ (see \cite[Section 9.1.1]{hh} for more details).
\end{dfn}

Assume that $I$ is an ideal of $S$ and ${\rm Min}(I)$ is the set of minimal primes of $I$. For every integer $k\geq 1$, the $k$-th {\it symbolic power} of $I$,
denoted by $I^{(k)}$, is defined to be$$I^{(k)}=\bigcap_{\frak{p}\in {\rm Min}(I)} {\rm Ker}(R\rightarrow (R/I^k)_{\frak{p}}).$$Let $I$ be a squarefree monomial ideal with irredundant
primary decomposition $$I=\frak{p}_1\cap\ldots\cap\frak{p}_r,$$ where every
$\frak{p}_i$ is a prime ideal generated by a subset of the variables. It follows from \cite[Proposition 1.4.4]{hh} that for every integer $k\geq 1$, $$I^{(k)}=\frak{p}_1^k\cap\ldots\cap
\frak{p}_r^k.$$In particular, for every graph $G$, we have $$J(G)^{(k)}=\bigcap_{\{x_i, x_j\}\in E(G)}(x_i, x_j)^k,$$for every integer $k\geq 1$.

For a monomial ideal $I$, we denote the set of its minimal monomial generators by $G(I)$. The {\it degree} of $I$, denoted by ${\rm deg}(I)$ is the maximum degree of elements of $G(I)$. Thus, in particular, ${\rm deg}(J(G))$ is the cardinality of the largest minimal vertex cover of the graph $G$.


\section{Main Results} \label{sec3}

In this section, we prove the main results of this paper. Namely, we show in Theorem \ref{reg} that for certain classes of graphs, including bipartite graphs, unmixed graphs and claw--free graphs, the inequalities \ref{1} hold. The first inequality is indeed true if one replaces $J(G)$ by any arbitrary squarefree monomial ideal. The proof of this assertion is simple and it follows immediately from the following lemma.

\begin{lem} \label{deg}
Let $I$ be a squarefree monomial ideal of $S$. For every integer $k\geq 1$, we have ${\rm deg}(I^{(k)})\geq k{\rm deg}(I)$.
\end{lem}

\begin{proof}
Let $I=\frak{p}_1\cap\ldots\cap\frak{p}_r,$ be the irredundant
primary decomposition of $I$. Choose a squarefree monomial $u\in G(I)$ with ${\rm deg}(u)={\rm deg}(I)$. Notice that $u^k\in I^k\subseteq I^{(k)}$. As $u\in G(I)$, for every $1\leq i\leq n$, there exists an integer $1\leq j\leq r$ such that $u/x_i\notin \frak{p}_j$. Thus $(u/x_i)^k\notin \frak{p}_j$. Since $\frak{p}_j^k$ is a $\frak{p}_j$-primary ideal and $x_i^{k-1}\notin \frak{p}_j^k$, we conclude that $u^k/x_i=(u/x_i)^kx_i^{k-1}\notin \frak{p}_j^k$. Consequently, $u^k/x_i\notin I^{(k)}$, for every integer $i$ with $1\leq i\leq n$. Thus, $u^k$ belongs to the set of minimal monomial generators of $I^{(k)}$. Hence,$${\rm deg}(I^{(k)})\geq {\rm deg}(u^k)=k{\rm deg}(I).$$
\end{proof}

We are now ready to prove the first main result of this paper.

\begin{thm} \label{reg}
Let $\mathcal{H}$ be a family of graphs which satisfies the following conditions.
\begin{itemize}
\item[(i)] For every graph $G\in \mathcal{H}$ and every vertex $x\in V(G)$, the graph $G\setminus N_G[x]$ belongs to $\mathcal{H}$.
\item[(ii)] If $G\in \mathcal{H}$ has no isolated vertex, then it admits a minimal vertex cover with cardinality at least $\frac{|V(G)|}{2}$.
\end{itemize}
Then for every graph $G\in \mathcal{H}$ and every integer $k\geq 1$, we have$$k{\rm deg}(J(G))\leq {\rm reg}(J(G)^{(k)})\leq (k-1){\rm deg}(J(G))+|V(G)|-1.$$
\end{thm}

\begin{proof}
The first inequality is an immediate consequence of Lemma \ref{deg}. Therefore, we prove the second inequality. Equivalently, we prove that$${\rm reg}(S/J(G)^{(k)})\leq (k-1){\rm deg}(J(G))+n-2,$$where $n=|V(G)|$. By replacing $\mathcal{H}$ with $\mathcal{H}\cup\{K_2\}$, we may assume that $K_2\in \mathcal{H}$. Let $m$ be the number of edges of $G$. We prove the assertions by induction on $m+k$.

By (i), for every graph $G\in \mathcal{H}$, the graph obtained from by $G$ deleting its isolated vertices belongs to $\mathcal{H}$. Thus, we can assume that $G$ has no isolated vertex. The assertion is well-known for $k=1$ (it follows, for example, by looking at the Taylor resolution of $J(G)$). If $m=1$, then $G=K_2$. In this case $J(G)=(x_1,x_2)$. Hence, ${\rm deg}(J(G))=1$ and ${\rm reg}(J(G)^{(k)})=k$. Thus, the desired inequality is true for $m=1$. Therefore, assume that $k,m\geq 2$. Let $S_1=\mathbb{K}[x_2, \ldots, x_n]$ be the polynomial ring obtained from $S$ by deleting the variable $x_1$ and consider the ideals $J_1=J(G)^{(k)}\cap S_1$ and
$J_1'=(J(G)^{(k)}:x_1)$. It follows from \cite[Lemma 2.10]{dhs} that
\[
\begin{array}{rl}
{\rm reg}(S/J(G)^{(k)})\leq \max \{{\rm reg}_{S_1}(S_1/J_1), {\rm reg}_S(S/J_1')+1\},
\end{array} \tag{2} \label{2}
\]

Set $u_1=\prod_{x_j\in N_G(x_1)}x_j\in S_1$. Hence, ${\rm deg}(u_1)={\rm deg}_G(x_1)$ and by \cite[Lemma 2.2]{s4},$$J(G)\cap S_1=u_1J(G\setminus N_G[x_1])S_1.$$It then follows that$$J_1=J(G)^{(k)}\cap S_1=(J(G)\cap S_1)^{(k)}=u_1^kJ(G\setminus N_G[x_1])^{(k)}S_1.$$Notice that if $C$ is a minimal vertex cover of $G\setminus N_G[x_1]$, then $C\cup N_G(x_1)$ is a minimal vertex cover of $G$. This shows that$${\rm deg}(J(G\setminus N_G[x_1]))+{\rm deg}_G(x_1)\leq {\rm deg}(J(G)).$$On the other hand, \cite[Lemma 4.1]{s3} implies that ${\rm reg}(J(G)\cap S_1)\leq {\rm reg}(J(G))$ and therefore,$${\rm reg}_{S_1}(S_1/J(G\setminus N_G[x_1])S_1)\leq {\rm reg}(S/J(G))-{\rm deg}(u_1).$$Since $G\setminus N_G[x_1]\in \mathcal{H}$, the induction hypothesis implies that
\begin{align*}
& {\rm reg}_{S_1}(S_1/J_1)={\rm reg}_{S_1}(S_1/J(G\setminus N_G[x_1])^{(k)}S_1)+k{\rm deg}(u_1)\\
& \leq (k-1){\rm deg}(J(G\setminus N_G[x_1]))+|V(G\setminus N_G[x_1])|-2+k{\rm deg}_G(x_1)\\
& \leq (k-1)({\rm deg}(J(G))-{\rm deg}_G(x_1))+n-{\rm deg}_G(x_1)-1-2+k{\rm deg}_G(x_1)\\
& < (k-1){\rm deg}(J(G))+n-2.
\end{align*}
Thus, using the inequality (\ref{2}), it is enough to prove that$${\rm reg}_S(S/J_1')\leq (k-1){\rm deg}(J(G))+n-3.$$

For every integer $i$ with $2\leq i\leq n$, let $S_i=\mathbb{K}[x_1, \ldots, x_{i-1}, x_{i+1}, \ldots, x_n]$ be the polynomial ring obtained from $S$ by deleting the variable $x_i$ and consider the ideals $J_i'=(J_{i-1}':x_i)$ and $J_i=J_{i-1}'\cap S_i$.

{\bf Claim.} For every integer $i$ with $1\leq i\leq n-1$ we have$${\rm reg}(S/J_i')\leq \max\{(k-1){\rm deg}(J(G))+n-3, {\rm reg}_S(S/J_{i+1}')+1\}.$$

\vspace{0.4cm}
{\it Proof of the Claim.} For every integer $i$ with $1\leq i\leq n-1$, we know from \cite[Lemma 2.10]{dhs} that
\[
\begin{array}{rl}
{\rm reg}(S/J_i')\leq \max \{{\rm reg}_{S_{i+1}}(S_{i+1}/J_{i+1}), {\rm reg}_S(S/J_{i+1}')+1\}.
\end{array} \tag{3} \label{3}
\]

Notice that for every integer $i$ with $1\leq i\leq n-1$, we have $J_i'=(J(G)^{(k)}:x_1x_2\ldots x_i)$. Thus,$$J_{i+1}=J_i'\cap S_{i+1}=((J(G)^{(k)}\cap S_{i+1}):_{S_{i+1}}x_1x_2\ldots x_i).$$Hence, it follows from \cite[Lemma 4.2]{s3} that

\[
\begin{array}{rl}
{\rm reg}_{S_{i+1}}(S_{i+1}/J_{i+1})\leq {\rm reg}_{S_{i+1}}(S_{i+1}/(J(G)^{(k)}\cap S_{i+1})).
\end{array} \tag{4} \label{4}
\]

Set $u_{i+1}=\prod_{x_j\in N_G(x_{i+1})}x_j\in S_{i+1}$. By Lemma \cite[Lemma 2.2]{s4},$$J(G)\cap S_{i+1}=u_{i+1}J(G\setminus N_G[x_{i+1}])S_{i+1}.$$Therefore,$$J(G)^{(k)}\cap S_{i+1}=(J(G)\cap S_{i+1})^{(k)}=u_{i+1}^kJ(G\setminus N_G[x_{i+1}])^{(k)}S_{i+1}.$$Notice that if $C$ is a minimal vertex cover of $G\setminus N_G[x_{i+1}]$, then $C\cup N_G(x_{i+1})$ is a minimal vertex cover of $G$. This shows that$${\rm deg}(J(G\setminus N_G[x_{i+1}]))+{\rm deg}_G(x_{i+1})\leq {\rm deg}(J(G)).$$On the other hand, \cite[Lemma 4.1]{s3} implies that ${\rm reg}(J(G)\cap S_{i+1})\leq {\rm reg}(J(G))$. Therefore,$${\rm reg}_{S_{i+1}}(S_{i+1}/J(G\setminus N_G[x_{i+1}])S_{i+1})\leq {\rm reg}(S/J(G))-{\rm deg}(u_{i+1}).$$Since $G\setminus N_G[x_{i+1}]$ belongs to $\mathcal{H}$, the induction hypothesis implies that
\begin{align*}
& {\rm reg}_{S_{i+1}}(S_{i+1}/(J(G)^{(k)}\cap S_{i+1}))={\rm reg}_{S_{i+1}}(S_{i+1}/J(G\setminus N_G[x_{i+1}])^{(k)}S_{i+1})+k{\rm deg}(u_{i+1})\\
& \leq (k-1){\rm deg}(J(G\setminus N_G[x_{i+1}]))+|V(G\setminus N_G[x_{i+1}])|-2+k{\rm deg}_G(x_{i+1})\\
& \leq (k-1)({\rm deg}(J(G))-{\rm deg}_G(x_{i+1}))+n-{\rm deg}_G(x_{i+1})-1-2+k{\rm deg}_G(x_{i+1})\\
& \leq (k-1){\rm deg}(J(G))+n-3.
\end{align*}
Finally, the claim now follows by inequalities (\ref{3}) and (\ref{4}).

\vspace{0.4cm}

Now, $J_n'=(J(G)^{(k)}:x_1x_2\ldots x_n)$ which is equal to $J(G)^{(k-2)}$ by \cite[Lemma 3.4]{s5}. Thus, by induction hypothesis we conclude that$${\rm reg}(S/J_n')\leq (k-3){\rm deg}(J(G))+n-2.$$ Therefore, using the claim repeatedly, implies that
\begin{align*}
& {\rm reg}(S/J_1')\leq \max\{(k-1){\rm deg}(J(G))+n-3, {\rm reg}_S(S/J_n')+n-1\}\\
& \leq \max\{(k-1){\rm deg}(J(G))+n-3, (k-3){\rm deg}(J(G))+n-2+n-1\}.
\end{align*}
As $G\in \mathcal{H}$, the assumptions imply that $G$ has minimal vertex cover with cardinality at least $n/2$. This means that $2{\rm deg}(J(G))\geq n$. Thus, the above inequalities imply that$${\rm reg}(S/J_1')\leq (k-1){\rm deg}(J(G))+n-3.$$This completes the proof of the theorem.
\end{proof}

The following example from \cite{gv} shows that not every graph satisfies the condition (ii) of Theorem \ref{reg}. However, we will see in Theorems \ref{bip}, \ref{unmixed} and \ref{claw} that the condition (ii) of Theorem \ref{reg} is satisfied by bipartite graphs, unmixed graphs and claw--free graphs.

\begin{exmp}
For every pair of integers $n\geq 3$ and $s\geq2$, let $G_{n,s}$ be the graph obtained by attaching $s$ pendant edges at each vertex of $K_n$. The graph $G_{3,3}$ is shown below. It is easy to check that the largest minimal vertex cover of $G_{n,s}$ has $n+s-1<\frac{|V(G_{n,s})|}{2}$ vertices (note that every vertex cover of $G$ contains at least $n-1$ vertices of $V(K_n)$).
\end{exmp}
\unitlength 1mm 
\linethickness{0.8pt}
\ifx\plotpoint\undefined\newsavebox{\plotpoint}\fi 
\begin{picture}(30,70)(35,-20) \label{fig1}
\put(81,10){\line(1,0){40}}
\put(81,10){\line(1,1){20}}
\put(101,30){\line(1,-1){20}}
\put(71,5){\line(2,1){10}}
\put(71,10){\line(2,0){10}}
\put(71,15){\line(2,-1){10}}

\put(101,30){\circle*{1}}
\put(81,10){\circle*{1}}
\put(121,10){\circle*{1}}

\put(71,5){\circle*{1}}
\put(71,10){\circle*{1}}
\put(71,15){\circle*{1}}

\put(131,5){\line(-2,1){10}}
\put(131,10){\line(-2,0){10}}
\put(131,15){\line(-2,-1){10}}

\put(131,5){\circle*{1}}
\put(131,10){\circle*{1}}
\put(131,15){\circle*{1}}

\put(106,40){\line(-1,-2){5}}
\put(101,40){\line(0,-2){10}}
\put(96,40){\line(1,-2){5}}

\put(106,40){\circle*{1}}
\put(101,40){\circle*{1}}
\put(96,40){\circle*{1}}

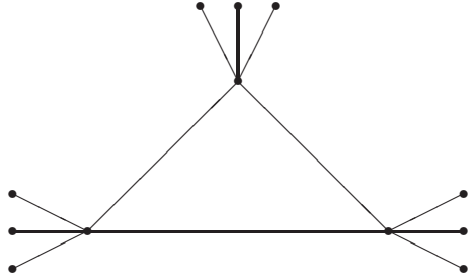
\put(78,-8){{\bf Figure 1.} The graph $G_{3,3}$}
\end{picture}

As we mentioned in the introduction, Hang and Trung \cite{ht1} recently proved the inequality ${\rm reg}(J(G)^k)\leq (k-1){\rm deg}(J(G))+|V(G)|-1$, for every bipartite graph $G$ and every integer $k\geq |V(G)|+2$. The following Theorem shows that this inequality is indeed true for every positive integer $k$.

\begin{thm} \label{bip}
Let $G$ be a bipartite graph. Then for every integer $k\geq 1$, we have$$k{\rm deg}(J(G))\leq {\rm reg}(J(G)^k)\leq (k-1){\rm deg}(J(G))+|V(G)|-1.$$
\end{thm}

\begin{proof}
Let $\mathcal{H}$ be the family of all bipartite graphs. It is clear that for every graph $G\in \mathcal{H}$ and every vertex $x\in V(G)$, we have $G\setminus N_G[x]\in \mathcal{H}$.

Let $G$ be a bipartite graph without isolated vertices. Assume that $V(G)=A\cup B$ is a bipartition for the vertex set of $G$. Without loss of generality, we may suppose that $|A|\geq |B|$. Then $A$ is a minimal vertex cover of $G$ with cardinality at least $\frac{|V(G)|}{2}$. On the other hand, it follows from \cite[Corollary 2.6]{grv} that for every integer $k\geq 1$ we have $J(G)^k=J(G)^{(k)}$. The desired inequalities now follow from Theorem \ref{reg}.
\end{proof}

The following corollary shows that the inequalities of Theorem \ref{bip} are sharp.

\begin{cor} \label{eq1}
For the complete bipartite graph $K_{1,n}$, we have$${\rm reg}(J(K_{1,n})^k)=kn,$$for every integer $k\geq 1$.
\end{cor}

\begin{proof}
The assertion follows from Theorem \ref{bip} by noticing that ${\rm deg}(J(K_{1,n}))=n$.
\end{proof}
Let $G$ be a bipartite graph. In \cite[Theorem 4.3]{s3}, we proved that $k{\rm deg}(J(G))+{\rm reg}(J(G))-1$ is an upper bound for the regularity of $J(G)^k$. In the same paper, \cite[Remark 4.4]{s3}, we mentioned that this bound is not probably the best one. In fact the upper bound of Theorem \ref{bip} is an improvement for the bound given by \cite[Theorem 4.3]{s3}. To see this, assume that $G$ has no isolated vertex and suppose that $V(G)=A\cup B$ is a bipartition for the vertex set of $G$. Without loss of generality, assume that $|A|\geq |B|$. As $A$ is a minimal vertex cover of $G$, we conclude that $|A|\leq {\rm deg}(J(G))$ and hence, $|A|\leq {\rm reg}(J(G))$. Thus,

\begin{align*}
(k-1){\rm deg}(J(G)) & +|V(G)|-1\leq (k-1){\rm deg}(J(G))+2|A|-1\\ & \leq k{\rm deg}(J(G))+{\rm reg}(J(G))-1.
\end{align*}

The second class of graphs which we consider is the class of unmixed graphs.

\begin{thm} \label{unmixed}
Let $G$ be an unmixed graph. Then for every integer $k\geq 1$, we have$$k{\rm deg}(J(G))\leq {\rm reg}(J(G)^{(k)})\leq (k-1){\rm deg}(J(G))+|V(G)|-1.$$
\end{thm}

\begin{proof}
Let $\mathcal{H}$ be the family of all unmixed graphs. Assume that $G$ is an unmixed graph and $x$ is an arbitrary vertex of $G$. Then a subset $W\subseteq V(G\setminus N_G[x])$ is a minimal vertex cover of $G\setminus N_G[x]$ if and only if $W\cup N_G(x)$ is a minimal vertex cover of $G$. Hence $G\setminus N_G[x]\in \mathcal{H}$. On the other hand, we know from \cite{gv} (see also \cite[Theorem 0.1]{crt}) that if $G$ is an unmixed graph without isolated vertices, then it has a minimal vertex cover with cardinality at least $\frac{|V(G)|}{2}$. The desired inequalities now follow from Theorem \ref{reg}.
\end{proof}

The last class of graphs which we study is the family of claw--free graphs.

\begin{thm} \label{claw}
Let $G$ be a claw--free graph. Then for every integer $k\geq 1$, we have$$k{\rm deg}(J(G))\leq {\rm reg}(J(G)^{(k)})\leq (k-1){\rm deg}(J(G))+|V(G)|-1.$$
\end{thm}

\begin{proof}
Let $\mathcal{H}$ be the family of all claw--free graphs. It is clear that for every graph $G\in \mathcal{H}$ and every vertex $x\in V(G)$, we have $G\setminus N_G[x]\in \mathcal{H}$. The assertion now follows from Theorem \ref{reg} together with the following claim.

{\bf Claim.} If $G$ is a claw--free graph which has no isolated vertex, then it admits a minimal vertex cover with cardinality at least $\frac{|V(G)|}{2}$.

{\it Proof of the Claim.} We use induction on $|V(G)|$. There is nothing to prove for $|V(G)|=2, 3$. Therefore, suppose that $|V(G)|\geq 4$. Without loss of generality, we may assume that $G$ is a connected graph. Let $W$ be the subset of vertices of $G$ with degree at least two. Assume that there is a vertex $x\in W$ which is adjacent to at least two leaves, say $y,z$. As $G$ is connected and claw--free, we conclude that $G$ has no other vertex, i.e., $V(G)=\{x,y,z\}$, which contradicts our assumption that $|V(G)|\geq 4$. Thus, every vertex in $W$ is adjacent to at most one leaf. If every vertex in $W$ is adjacent to exactly one leaf, then $|V(G)|=2|W|$ and $W$ is a minimal vertex cover of $G$. Hence, the claim follows in this case.

Therefore, assume that there is a vertex $v\in W$ which is adjacent to no leaf in $G$. Let $w$ be a neighborhood of $v$. As $w$ is not a leaf, it follows that  $w\in W$. Set $H=G\setminus N_G[w]$ and suppose that $U$ is the set of isolated vertices of $H$. Then every vertex in $U$ is adjacent to a vertex in $N_G(w)$. If there are to vertices $w_1, w_2\in U$ which are adjacent to a vertex $w_0\in N_G(w)$, then the vertices $w, w_0, w_1, w_2$ form a claw which is a contradiction. Thus, every vertex in $N_G(w)$ is adjacent to at most one vertex in $U$. This shows that $|U|\leq |N_G(w)|$. However, we prove the following stronger inequality. \[
\begin{array}{rl}
|U|\leq |N_G(w)|-1
\end{array} \tag{5} \label{5}
\]
Indeed, using the above argument, the inequality (\ref{5}) is obvious, if there is a vertex in $U$ which is adjacent to at at least two vertices in $N_G(w)$. Hence, suppose that every vertex in $U$ is adjacent to exactly one vertex in $N_G(w)$. Thus, the vertices of $W'$ have degree one in $G$. This shows that no vertex in $U$ is adjacent to $v$. Since $v\in N_G(w)$, again the above argument implies that $|U|\leq |N_G(w)|-1$.

Note that $H\setminus U$ is a claw--free graph which has no isolated vertex. It follows from the induction hypothesis that $H\setminus U$ has a minimal vertex cover $C$ with$$|C|\geq \frac{|V(H\setminus U)|}{2}=\frac{|V(H)|-|U|}{2}.$$Then $C\cup N_G(w)$ is a minimal vertex cover of $G$ and
\begin{align*}
|C\cup N_G(w)| & =|C|+|N_G(w)|\geq \frac{|V(H)|-|U|}{2}+ |N_G(w)|\\ & \geq \frac{|V(H)|+|N_G(w)|+1}{2}=\frac{|V(G)|}{2},
\end{align*}
where the last inequality follows from the inequality (\ref{5}).
\end{proof}

The following corollary shows that the inequalities of Theorem \ref{claw} are sharp.

\begin{cor} \label{eq2}
Assume that $H$ is a graph with $i(G)\leq 2$. Let $G$ be the graph obtained from $H$ by adding a new vertex $y$ and connecting $y$ to every vertex of $H$. Then for every integer $k\geq 1$, we have$$k{\rm deg}(J(G))={\rm reg}(J(G)^{(k)})=(k-1){\rm deg}(J(G))+|V(G)|-1.$$In particular,$${\rm reg}(J(K_n)^{(k)})=k(n-1),$$for every integer $n\geq 2$.
\end{cor}

\begin{proof}
It is obvious from the construction of $G$ that $i(G)\leq 2$. Hence, $G$ is a claw--free graph. On the other hand $V(H)$ is a minimal vertex cover of $G$. Thus, ${\rm deg}(J(G))=|V(H)|=|V(G)|-1$. The assertions now follow from Theorem \ref{claw}.
\end{proof}





\end{document}